\documentclass[12pt,a4paper,psamsfonts]{amsart}
\usepackage{array}
\usepackage[english]{babel}
\usepackage[T1]{fontenc}
\usepackage{fancyhdr}
\usepackage{appendix}
\usepackage{amssymb,amscd,amsxtra,calc}
\usepackage{booklet}
\usepackage{mathrsfs}
\usepackage{cmmib57}
\usepackage{multirow}
\usepackage[all]{xy}
\usepackage{longtable}
\usepackage[colorlinks=true,linkcolor=blue,anchorcolor=blue,citecolor=blue,pagebackref,linktocpage]{hyperref}
\usepackage{cleveref}
\usepackage{tikz}
\usepackage{cite}
\theoremstyle{plain}
    \newtheorem{thm}{Theorem}[section]

    \newtheorem{example}[thm]{Example}
    \newtheorem{lemma}[thm]{Lemma}
    \newtheorem{proposition}[thm]{Proposition}
    
    \newtheorem{theorem}[thm]{Theorem}

\theoremstyle{definition}
    \newtheorem{definition}[thm]{Definition}

    \newtheorem*{notation*}{Notation and Terminology}
      
    \newtheorem{remark}[thm]{Remark}
    
\theoremstyle{remark}

    \newtheorem{setup}[thm]{}

\newcommand{\bP}{\mathbb{P}}
\newcommand{\bQ}{\mathbb{Q}}
\newcommand{\bR}{\mathbb{R}}

\newcommand{\bZ}{\mathbb{Z}}

\newcommand{\Amp}{\operatorname{Amp}}

\newcommand{\Aut}{\operatorname{Aut}}

\newcommand{\Contr}{\operatorname{Contr}}

\newcommand{\Exc}{\operatorname{Exc}}

\newcommand{\id}{\operatorname{id}}

\newcommand{\NE}{\overline{\operatorname{NE}}}
\newcommand{\Nef}{\operatorname{Nef}}

\newcommand{\NS}{\operatorname{NS}}

\newcommand{\Proj}{\operatorname{Proj}}

\newcommand{\N}{\operatorname{N}}

\newcommand{\Pic}{\operatorname{Pic}}

\newcommand{\mstriangle}[1]{
\begin{tikzpicture}[x=0.3cm,y=0.3cm]
\draw (-0.4,-0.433) -- (1.4,-0.433);
\draw (-0.2,-0.7794) -- (0.7,0.7794);
\draw (1.2,-0.7794) -- (0.3,0.7794);
\end{tikzpicture}
}
\newcommand{\mssharp}[1]{
\begin{tikzpicture}[x=0.3cm,y=0.3cm]
\draw (-0.8,-0.5) -- (0.8,-0.5);
\draw (-0.8,0.5) -- (0.8,0.5);
\draw (-0.5,-0.8) -- (-0.5,0.8);
\draw (0.5,-0.8) -- (0.5,0.8);
\end{tikzpicture}
}

\makeatletter
\renewcommand{\@maketitle}{%
  \normalfont\normalsize
  \@adminfootnotes
  \@mkboth{\@nx\shortauthors}{\@nx\shorttitle}%
  \global\topskip42\p@\relax
  \@settitle
  \ifx\@empty\@dedicatory\else
    \baselineskip18\p@
    \vtop{\centering
      {\normalsize\itshape\@dedicatory\@@par}%
      \global\dimen@i\prevdepth}%
    \prevdepth\dimen@i
  \fi
  \ifx\@empty\authors\else
    \@setauthors
  \fi
  \@setabstract
  \normalsize
  \if@titlepage
    \newpage
  \else
    \dimen@34\p@
    \advance\dimen@-\baselineskip
    \vskip\dimen@\relax
  \fi
}
\makeatother

\begin{document}
\title[Free abelian group actions on projective fibered varieties]{Free abelian group actions on projective fibered varieties}
\dedicatory{Dedicated to Professor Meng Chen
on the occasion of his 60th birthday}
\author{Sichen Li}
\address{
School of Mathematics, East China University of Science and Technology, Shanghai 200237, P. R. China}
\email{\href{mailto:sichenli@ecust.edu.cn}{sichenli@ecust.edu.cn}}
\begin{abstract}
Let $X$ be an $n$-dimensional normal projective variety of Kodaira dimension $\kappa$ and $G$ an abelian subgroup of automorphisms such that all elements of $G\backslash \{\id\}$ are of positive entropy.
Dinh and Sibony showed that $G$ is actually free abelian of rank $\le n-\kappa-1$.
In this paper, we show that after replacing $G$ by a finite-index subgroup, there exists  a G-Equivariant Minimal Model Program for the pair $(X, G)$ when $\mathrm{rank}(G)=n-\kappa-1$ and $\kappa>0$, and the very general fiber of the final model is a  klt Calabi-Yau  surface or a $Q$-abelian variety provided that the termination of flips holds in this setting.
\end{abstract}



\maketitle

\section{Introduction} 
Let $f: X\to  X$ be a surjective endomorphism of an $n$-dimensional normal projective variety $X$. 
Then the {\it first dynamical degree} of $f$ (cf. \cite{Dang20, DS05, Truong20}) is defined as 
$$
 d_1(f):=\lim_{s\to\infty}((f^s)^*H\cdot H^{n-1})^{1/s},
$$
where $H$ is an ample Cartier divisor on $X$.
Dinh and Sibony's result states that this limit exists and is independent of the choice of the ample divisor $H$.
We say $f$ is of positive entropy if $d_1(f)>1$, otherwise $f$ is of zero entropy.
We say $f$ is {\it imprimitive} (cf. \cite{Zhang09-JDG}) if there exists a dominant rational map $\pi: X\dasharrow Y$ and a dominant rational self-map $f|_Y: Y\dasharrow Y$ such that $\pi\circ f=f|_Y\circ\pi$.
We say such $\pi$ is $f$-equivariant.
We say $f$ is {\it primitive} if it is not imprimitive.

Let $X$ be an $n$-dimensional normal projective variety with the Kodaira dimension $\kappa(X)$ and $\Aut(X)$ the automorphism group of  $X$.
It is well-known that every abelian subgroup $G\subseteq \Aut(X)$ of positive entropy is free of dynamical rank $r(G)\le \dim X-\kappa(X)-1$ (cf. \cite{CWZ14, DS04, Zhang09}, see also Theorem \ref{Zhang09-thm} and Definition \ref{rank-defn}).
Since then, algebraic dynamics on systems with maximal dynamical rank have been intensively studied (cf. \cite{DHZ15, Dinh12, Hu20, HTZ18, Zhang13, Zhang14, Zhang16}), and the sub-maximal dynamical rank case was first studied by Hu and the author \cite{HL21}.
Thus, it is meaningful for us to consider the following hypothesis.
\begin{definition}
\label{hyp-defn}
Let $X$ be an $n$-dimensional normal projective variety with Kodaira dimension $\kappa$, and let $G\le\operatorname{Aut}(X)$. Suppose $X$ has at worst $\mathbb{Q}$-factorial klt singularities. We say $(X,G)$ satisfies $\operatorname{Hyp}(n,r)$ if $G\cong\mathbb{Z}^r$ with $1\le r\le n-\kappa-1$,   every non-trivial element of $G$ has positive entropy.
\end{definition}
A basic tool for characterizing $(X, G)$ is to run the G-Equivariant Minimal Model Program (G-EMMP) on a mildly singular $X$ as illustrated by the following commutative diagram:
\[
\xymatrix{
X=X_1 \ar@{-->}[r]^{\quad\pi_1} \ar@{-->}[d]^{f_1:=f^s}  &   X_2 \ar@{-->}[d]^{f_2}  \ar@{-->}[r]^{\pi_2} & \cdots \ar@{-->}[r]^{\pi_{n-1}} & X_n \ar@{-->}[d]^{f_n} \\
  X=X_1 \ar@{-->}[r]^{\quad\pi_1}  & X_2 \ar@{-->}[r]^{\pi_2} & \cdots \ar@{-->}[r]^{\pi_{n-1}} & X_n\\
}
\]
for every $f\ne\id\in G$ and some positive integer $s$ such that $f_i\in G|_{X_i}$ for each $i$.
Here, $\pi_i$ is either a divisorial contraction, a flip, or a Fano contraction of a $K_{X_i}$-negative extremal ray.
After iterating $f_i$, a Fano contraction is always equivariant, while a divisorial contraction (resp. flip) is equivariant if one can show the exceptional locus (resp. flipping locus) is $f_i$-periodic by \cite[Lemma 6.2]{MZ18}.

In this paper, we will clarify $(X,G)$ which satisfies $\mathrm{Hyp}(n,n-\kappa-1)$.
When $n=2$, the classification of $(X, G)$ is a theorem of Cantat \cite{Cantat99}.
When $\kappa\le 0$,  Zhang's maximal dynamical rank theorem \cite{Zhang16} proved that if  $\dim X\ge3$ and $X$ is not rationally connected, then $X$ is $G$-equivariantly birational to a $Q$-abelian variety (cf. \cite[Definition 2.13]{NZ10}).
When $\kappa=1$, Hu and the author showed in \cite[Lemma 3.1]{HL21} that the general fiber of the Iitaka fibration is $G$-equivariantly birational to a K3 surface, an Enriques surface, or a $Q$-abelian variety.
In general, when $\kappa>0$ and assume that the termination of flips holds, we may run a G-EMMP for $(X, G)$, and show that the very general fiber of the Iitaka fibration of the final model is a klt Calabi-Yau surface or a $Q$-abelian variety, which is our main result in this paper as follows.
\begin{theorem}
\label{Relative-MDR}
Suppose $(X,G)$ satisfies $\mathrm{Hyp}(n, n-\kappa(X)-1),\dim X\ge3, \kappa(X) >0$ and the termination of flips holds.
Then after replacing $G$ by a finite-index subgroup, there exists a G-EMMP such that the final model $(X', G')$ is relatively minimal over the base space of the Iitaka fibration $\Phi:=\varphi|_{mK_{X'}}$, and the very general fiber of $\Phi$ is a klt Calabi-Yau surface or a $Q$-abelian variety.
\end{theorem} 
\begin{example}
Let $X:=Y\times F$ be the product of  $Y$ and $F$, where $Y$ is a minimal projective variety of general type with $\dim Y=\kappa>0$, and $F$ is a K3 surface, an Enriques surface, or an abelian variety, and $(F,G|_F)$ satisfies $\mathrm{Hyp}(n-\kappa,n-\kappa-1)$ (cf. \cite{McMullen11}, \cite{OY20}, \cite[Example 1.4]{Dinh12}).
Let $G:=\{\id_Y\}\times G|_F$.
Therefore, $(X,G)$ satisfies $\mathrm{Hyp}(n,n-\kappa-1)$.
\end{example}
\begin{remark}
Let $(X,G)$ satisfy $\mathrm{Hyp}(n,n-\kappa-1)$ and $\kappa>0$.
By Theorem \ref{Relative-MDR} and Proposition \ref{Bir}, one may assume that $X$ is relatively minimal.
However, it is unknown whether there exists an example of the Iitaka fibration of $(X, G)$ that is not isotrivial.
For instance, is it isotrivial for the Iitaka fibration of every projective threefold $X$  of $\kappa=1$  admitting an automorphism $f$ with $d_1(f)>1$?
\end{remark}
 {\bf Acknowledgment.}
 I thank Meng Chen for his constant encouragement, Rong Du and
Guolei Zhong for helpful conversations, Haidong Liu for answering my questions, and
the referees for their careful reading of the manuscript and many valuable suggestions.
The research is supported by Shanghai Sailing Program (No. 23YF1409300).

\section{Preliminaries}
\label{Prelim}
\subsection{Notation and conventions}
We refer to \cite{Fulton98, Kawamata24, KM98} for the standard definitions, notation, and terminology in birational geometry.

Let $X$ be a normal projective variety.
$\N_\bZ^k(X)$ is the group of codimension $k$ algebraic cycles on $X$ modulo numerical equivalence.
We will use the vector spaces $\N_\bR^k(X):=\N_\bZ^k(X)\otimes_\bZ\bR$.
In particular, $\N_\bR^1(X)=\NS(X)\otimes_\bZ\bR$, where $\NS(X):=\Pic(X)/\Pic^0(X)$ is the usual N\'eron-Severi group of $X$.

Let $\pi\colon X\to S$ be a projective morphism of algebraic varieties defined over a field $k$.
Two $\mathbb{R}$-divisors $D,D'$ on $X$ are said to be \emph{relatively $\mathbb{R}$-linearly equivalent}, denoted $D\sim_{\mathbb{R},S} D'$,
if the difference $D-D'$ satisfies
\[
D - D' \sim_{\mathbb{R}} \pi^*(H_S)
\]
for some $\mathbb{R}$-Cartier divisor $H_S$ on the base variety $S$.
A closed subvariety $Z\subset X$ is called a \emph{relative subvariety} over $S$ whenever $\pi(Z)$ is a closed point of $S$.
In particular, if $\dim Z = 1$, then $Z$ is referred to as a \emph{relative curve} over $S$.

Two Cartier divisors $D$ and $D'$ on $X$ are said to be \emph{relatively numerically equivalent}, denoted $D \equiv_S D'$, if $(D \cdot C) = (D' \cdot C)$ for every relative curve $C$ over $S$. 
This definition extends naturally to $\mathbb{R}$-Cartier divisors, with the same notation $D \equiv_S D'$ adopted for real relative numerical equivalence classes.
The abelian group of isomorphism classes of invertible sheaves on $X$ is written $\operatorname{Pic}(X)$. 
Its subgroup consisting of invertible sheaves relatively numerically equivalent to $\mathcal{O}_X$ is denoted $\operatorname{Pic}^\tau(X/S)$. The quotient group $\operatorname{Pic}(X)/\operatorname{Pic}^\tau(X/S)$ is a finitely generated abelian group (see \cite[p. 323]{Kleiman66}), called the \emph{relative N\'eron-Severi group} $\operatorname{NS}(X/S)$.
For any $\mathbb{R}$-Cartier divisor $D$, we write $[D]$ for its numerical equivalence class. The set of all such real numerical classes coincides with $\operatorname{NS}(X/S) \otimes \mathbb{R}$, which is  a real vector space of dimension $\rho(X/S)$, and is  denoted by $N^1(X/S)$.
The closure of the cone of all effective divisors in $N^1(X/S)$  is denoted by $\mathrm{Eff}(X/S)$, which is called the {\it relative pseudoeffective cone}.

The set of interior points of the closed convex cone $\mathrm{Eff}(X/S)$ is called the {\it relative big cone} and is denoted by $\mathrm{Big}(X/S)$.
An $\bR$-Cartier divisor $D$ is {\it relatively big} if its numerical equivalence class $[D]$ is contained in $\mathrm{Big}(X/S)$.

An $\bR$-Cartier divisor $D$ is said to be {\it relatively nef} if $(D\cdot C)\ge0$ for any relative curve $C$.
The set of numerical equivalence classes of all nef $\bR$-Cartier divisors is a closed convex cone in $N^1(X/S)$, which is denoted by $\Nef(X/S)$ and called the {\it relatively nef cone}.

The set of interior points of the relative nef cone is called the {\it relative ample cone} and is denoted by $\Amp(X/S)$.
An $\bR$-Cartier divisor $D$ is said to be {\it relatively ample} if its numerical equivalence class is contained in $\Amp(X/S)$.

Fix an integer $t$, a finite formal linear sum of $t$-dimensional relative subvarieties $Z=\sum a_jZ_j$ is called a {\it relative $t$-cycle}.
Relative 1-cycles $C,C'$ are said to be {\it numerically equivalent}, denoted by $C\equiv_SC'$, if $(D\cdot C)=(D\cdot C')$ for any Cartier divisor $D$.
The set $\N_1(X/S)$ of all numerical equivalence classes of relative 1-cycles with real coefficients is a finite-dimensional real vector space.
In $\N_1(X/S)$, the closure of the cone of relative curves is called the {\it closed cone of relative curves}, which is denoted by $\NE(X/S)$.
\begin{definition}
A normal projective surface \(X\) is called a \emph{klt Calabi-Yau surface} if \(X\) has klt singularities and \(K_X\sim_{\mathbb Q}0\).
\end{definition}
The following result, a generalization of the Hodge index theorem, extends \cite[Lemma 2.2]{NZ10} to the singular setting. 
\begin{lemma}
\label{Hodge-lem}
Let $X$ be a normal projective variety of dimension $n\ge2$.
Suppose that an $R$-Cartier divisor $D$ satisfies the following two conditions:
\begin{enumerate}
	\item[(i)]\(D\cdot G\cdot L_1\cdots L_{n-2}\geq 0\) for any effective \(\mathbb R\)-divisor $G$ and every collection of nef $\mathbb R$-Cartier divisors $L_1,\ldots,L_{n-2}$;
\item[(ii)] $D\cdot H_1\cdots H_{n-1}=0$ for some nef and big $\mathbb R$-Cartier divisors $H_1,\ldots,H_{n-1}$.
\end{enumerate}
Then \(D\equiv 0\).
\end{lemma}
\begin{proof}
Fix an ample Cartier divisor \(A\) on \(X\).
Replacing $A$ by a sufficiently large multiple, we may assume $A$ is also effective.
Since \(H_1\) is big, there exist \(a_1>0\) and an effective \(\mathbb R\)-divisor \(E_1\) such that
$
a_1H_1\equiv A+E_1.
$
Thus, by (ii), we have
\[
0=D\cdot A\cdot H_2\cdots H_{n-1}
+D\cdot E_1\cdot H_2\cdots H_{n-1}.
\]
Both terms on the right are nonnegative by (i), we have
\[
D\cdot A\cdot H_2\cdots H_{n-1}=0.
\]
Repeating the argument for \(H_2,\ldots,H_{n-1}\), we obtain
$
D\cdot A^{n-1}=0.
$
\cite[Claim 3.3]{Zhang16} implies that $D^2\cdot A^{n-2}\le0$.
For \(m\gg0\), \(D+mA\) is big. 
Hence, there exists an effective \(\mathbb R\)-divisor \(\Delta\) such that $\Delta\equiv D+mA.$
By (i),
$
D\cdot\Delta\cdot A^{n-2}\ge0
$
implies that
\[
\begin{aligned}
D^2\cdot A^{n-2}
&=D\cdot\Delta\cdot A^{n-2}
-mD\cdot A^{n-1}\\
&\ge0.
\end{aligned}
\]
Thus, $D^2\cdot A^{n-2}=0$.
As a result,  $D\equiv0$ by \cite[Lemma 3.2]{Zhang16}.
\end{proof}
\subsection{Dynamical ranks}
A $G$-action on an algebraic variety $Y$ is a group homomorphism $G\to \Aut(Y)$.
We denote by $G|_Y$ the image of $G$ in $\Aut(Y)$.
The action of $G$ on $Y$ is {\it faithful} if $G\to\Aut(Y)$ is injective.
Let $G$ be a subgroup of the automorphism group $\Aut(X)$.
A rational map $\pi: X\dasharrow Y$ is called {$G$-equivariant} if for each $g_X\in G$, there is an automorphism $g_Y$ of $Y$ such that $\pi\circ g_X=g_Y\circ \pi$.
We hence denote by $G|_Y$ the image of $G$ in $\Aut(Y)$.
We say a normal projective variety $X$ is a $Q$-abelian variety if there is a quasi-\'etale morphism $\pi: A\to X$ such that $A$ is an abelian variety.

The following Tits alternative type theorem is due to Zhang \cite{Zhang09}.
\begin{theorem}
\label{Zhang09-thm}
Let $X$ be a normal projective variety of dimension $n\ge2$ and $G$ a subgroup of $\Aut(X)$.
Then one of the following two assertions holds.
\begin{enumerate}
\item $G$ contains a subgroup isomorphic to the non-abelian free group $\bZ*\bZ$.
\item There is a finite-index subgroup $G_1$ of $G$ such that the induced group $G_1|_{\NS_\bR(X)}$ is solvable and $Z$-connected.
Moreover, the zero entropy subset $N(G_1)$ of $G_1$ is a normal subgroup of $G_1$ and the quotient group $G_1/N(G_1)$ is  free abelian of rank $r\le n-1$. 
\end{enumerate}
\end{theorem}

In the second assertion of Theorem \ref{Zhang09-thm}, the rank of $G_1/N(G_1)$ is independent of the choice of $G_1$. Hence, it makes sense to think of this as an invariant of $G$.
Under Hypothesis $\operatorname{Hyp}(n,r)$ from Definition \ref{hyp-defn}, $G \cong \mathbb{Z}^r$ is abelian, so its restriction $G|_{N_{\mathbb{R}}^1(X)}$ is automatically abelian and thus solvable. With this prerequisite satisfied, we introduce the following definition.

\begin{definition}
\label{rank-defn}
Let $X$ be a normal projective variety of dimension $n$ and $G$ a Zariski-connected subgroup of $\Aut(X)$ such that $G|_{N^1_{\mathbb R}(X)}$ is solvable. Then the rank of the free abelian group $G/N(G)$ is called the dynamical rank of $G$, denoted by $r(G)$.
\end{definition}
The following assertion states that the dynamical rank is a birational invariant.
\begin{proposition}
\cite[Lemmas 4.2 and 4.4]{Hu20}
\label{dr-birational}
Let $\pi: X\dasharrow Y$ be a $G$-equivariant generically finite dominant rational map of normal projective varieties.
Then after replacing $G$ by a finite-index subgroup, $G|_{\N_\bR^1(X)}$ is solvable if and only if so is $G|_{\N_\bR^1(Y)}$.
Moreover, $r(G|_X)=r(G|_Y)$.
\end{proposition}
The following is a crucial inequality of dynamical ranks in this paper.
\begin{proposition}
\cite[Lemmas 4.1 and 4.3]{Hu20}
\label{prop-dr}
Let $\pi: X\dasharrow Y$ be a $G$-equivariant dominant rational map of normal projective varieties with $n=\dim X>\dim Y=m>0$.
Suppose that $G|_{\N_\bR^1(X)}$ is solvable.
Then so is $G|_{\N_\bR^1(Y)}$, and we have
$$
        r(G|_X)\le r(G|_Y)+n-m-1.
$$
In particular, $r(G|_X)=n-m-1$ provided that $r(G|_Y)=0$.
\end{proposition}
The following asserts that $\mathrm{Hyp}(n, r)$ is a birational invariant (cf.  \cite[Lemma 3.1]{Zhang16}).
\begin{proposition}
\label{Bir}
Assume that a group $G$ acts biregularly on both normal projective varieties $X_1$ and $X_2$ of the same dimension $n$. Let $\sigma: X_1\dashrightarrow X_2$ be a $G$-equivariant generically finite map.
\begin{enumerate}
\item[(1)] Suppose that $\sigma$ is a morphism or $\sigma$ is an isomorphism in codimension 1 with both $X_1$ and $X_2$ being $\mathbb Q$-factorial. 
If $(X_1,G)$ satisfies $\mathrm{Hyp}(n, r)$, then so does $(X_2, G)$.
\item[(2)] Suppose that $(X_a, G)$, for some $a\in\big\{1,2\big\}$, satisfies $\mathrm{Hyp}(n, r)$. Then, replacing $G$ by a finite-index subgroup, $(X_i,G)$ satisfies $\mathrm{Hyp}(n, r)$ for both $i$ in $\big\{1,2\big\}$.
\end{enumerate}	
\end{proposition}
\begin{proof}
Replacing $\mathrm{Hyp}(n,n-1)$ by $\mathrm{Hyp}(n,r)$, the proof uses 
 the same argument as in \cite[Lemma 3.1]{Zhang16}.
 \end{proof}
\section{Proof of Theorem \ref{Relative-MDR}}
\label{Sec-MDR}
\begin{theorem}
\label{thm-n-kappa=2}
Let $(X,G)$ satisfy $\mathrm{Hyp}(n, n-\kappa(X)-1), \kappa(X)>0$ and $n\ge3$.
After replacing $G$ by a finite-index subgroup, the following statements hold.
\begin{enumerate}
	\item Let $\pi:=\Phi_{|mK_X|}: X\dasharrow S \subseteq \bP(H^0(X, mK_X))$ be the Iitaka fibration of $X$, where $B$ is the image of $\Phi_{|mK_X|}$ for $m\gg0$.
	Then we may assume that $\pi: X\to S$ is a regular morphism and $G|_S=\{\id\}$.
	\item Let $F$ be a very general fiber of $\pi$. 
We may identify $G$ with $G|_F\le \Aut(F)$ and $(F, G|_F)$ satisfies $\mathrm{Hyp}(n-\kappa(X), n-\kappa(X)-1)$ and $\kappa(F)=0$.
\end{enumerate}
\end{theorem}
\begin{proof}
By the Deligne-Nakamura-Ueno theorem (\cite[Theorem 14.10]{Ueno75}), $G$ descends to a finite group $G|_S$ acting on $S$ biregularly.
Replacing $G$ by a finite-index subgroup, which does not change its dynamical rank by Proposition \ref{dr-birational}, we may assume that $G|_S=\{\id\}$.
Further, replacing $X$ and $S$ by $G$-equivariant resolutions of singularities of the graph $\Gamma_\pi$ of $\pi$ and of $S$, we may also assume that $\pi$ is a regular morphism and $X$ is smooth, since by Proposition \ref{Bir} the new pair $(X,G)$ still satisfies $\mathrm{Hyp}(n,n-\kappa(X)-1)$.
Note that $r(G|_S)=0$.
By Proposition \ref{prop-dr}, $0<r(G|_X)\le n-\kappa(X)-1$ implies that $\kappa(X)\le n-2$.
Thus, $\pi: X\to S$ is a nontrivial $G$-equivariant fibration.
Since $G$ acts trivially on the base, $G$ acts naturally and regularly on the very general fiber $F$ of $\pi$.
For any $g\in G$, let $g_F$ denote the induced automorphism of $g$ on $F$.
By the product formula (cf. \cite[Theorem 1.1]{DN11}), $d_1(g|_F)=d_1(g)>1$ if $g\neq\id$.
Therefore, $G$ acts faithfully (and also regularly) on $F$ so that we can identify $G$ with $G|_F\le \Aut(F)$ and $(F, G|_F)$ satisfies $\mathrm{Hyp}(n-\kappa(X), n-\kappa(X)-1)$.
Note that $F$, as a very general fiber of the Iitaka fibration, has $\kappa(F)=0$.
\end{proof}
Now we recall a result of Dinh and  Sibony \cite[Theorems~4.3 and 4.7]{DS04}.
\begin{proposition}
\cite[Lemma 3.4]{Zhang16}
\label{DS}
Let $X$ be a normal projective variety and $G\le\mathrm{Aut}(X)$.
Suppose that $G^*:=G|_{\N_{\mathbb R}^1(X)}\cong\mathbb Z^{\oplus r}$ and every non-trivial element of $G^*$ is of positive entropy.
Then there are nef $\mathbb R$-Cartier divisors $L_i$ ($1\le i\le r+1)$ such that $L_i$ is a common eigenvector of $G$ with
\begin{equation*}
g^*L_i=\chi_i(g)L_i~~(g\in G)
\end{equation*}
for some $\chi_i(g)\in\mathbb R_{>0}$, such that
\begin{equation*}
L_1\cdots L_{r+1}\ne 0
\end{equation*}
as an element of $N_{\bR}^{r+1}(X)$, and such that the homomorphism
\begin{equation*}\begin{split}
\varphi &: G\rightarrow (\mathbb R^r, +)\\& g\mapsto (\mathrm{log}~\chi_1(g), \cdots, \mathrm{log}~\chi_r(g))
\end{split}\end{equation*}
has the kernel equal to $N(G)$ and hence it induces an isomorphism from $G/N(G)\cong G^*$ onto a spanning lattice of $(\mathbb R^{r}, +)$.	
\end{proposition}
\begin{setup}
\label{Setup}
Now assume that $(X,G)$ satisfies $\mathrm{Hyp}(n,n-\kappa-1)$ and $\kappa>0$.
By Proposition \ref{DS}, there are nef $\bR$-divisors $L_i$ $(1\le i\le n-\kappa)$ on $X$ such that $L_i$ is a common eigenvector of $G$.
Let $F$ be the very general fiber of the Iitaka fibration of $\varphi|_{mK_X}: X \to S$.
Notice that for each $1\le i\le n-\kappa$,  $L_i\cdot F$ is a common eigenvector of $G|_F$.
Set $A:=L_1+\cdots +L_{n-\kappa}$.
Then $((L_1 +\cdots+L_{n-\kappa})|_F)^{n-\kappa}\ge (L_1\cdot F)\cdots (L_{n-\kappa}\cdot F)$ and the latter is non-zero in $H^{n-\kappa,n-\kappa}(F,\bR)=\bR$ by Proposition \ref{DS}, and indeed positive since $L_i$ is nef.
Thus, $A$ is a nef and relatively big $\bR$-Cartier divisor over $S$.
By the relative version of Kodaira's lemma (cf. \cite[Theorem 1.5.7(2)]{Kawamata24}), for an integer $k\gg1$, we have $A=A_k+E/k$ such that $A_k$ is a relatively ample $\bQ$-Cartier divisor and $E$  is a fixed effective $\bR$-Cartier divisor.
Let $H_S$ be an ample divisor on $S$.
By \cite[Proposition 1.45]{KM98}, $A_k^*:=A_k+n\pi^*H_B$ is ample for $n\gg0$.
 Take $B$ to be a general member of $|mA_k^*|$ for some $m\gg1$.
Let $A_k^*=\frac{1}{m}B$.
Notice that $A_k^*\sim_S A_k$.
Notice that $X$ has at worst $\bQ$-factorial klt singularities and $K_X$ is fixed by $G$.
Let $A^*=A_k^*+E/k$.
Then $(X, A^*)$ is klt.
Then $K_X+A^*$ is ample (big, nef) if and only if $K_X+A$ is relatively ample (big, nef).
So in the following G-EMMP over the base space $S$, we relabel $A^*$ as $A$, and we may assume that $(X, A)$ is klt since their difference $n\pi^*H_S$ is stabilized by $G$.
Hence, we prove the following lemma.
\end{setup}
\begin{lemma}
\label{lemma-A-big}
Suppose $(X,G)$ satisfies $\mathrm{Hyp}(n,n-\kappa-1)$ and $\kappa>0$.
Then we have:
\begin{enumerate}
	\item $A=L_1+\cdots+L_{n-\kappa}$ in \ref{Setup} is a relatively nef and big $\bR$-Cartier divisor.
	\item For integer $k\gg 1$, we have $A=A_k+E/k$ such that $A_k$ is a relatively ample $\bQ$-Cartier divisor and $E$ is a fixed effective $\bR$-Cartier divisor.
	\item After replacing $A_k$ by some $\mathbb{Q}$-linearly equivalent divisor $A_k' \sim_{\mathbb{Q}} A_k$, we set $A' := A_k' + E/k$, which satisfies $A' \equiv_S A$ over the base $S$ and still decomposes as a sum of nef $G$-eigenvectors. Replacing the original $A$ by $A'$, we may assume that $A$ is effective and the pair $(X,A)$ has at worst klt singularities.
	\end{enumerate}
\end{lemma}
\begin{theorem}
\label{thm-MMP}
Let $(X,G)$ satisfy $\mathrm{Hyp}(n,n-\kappa-1)$ and $\kappa>0$.
Let $\pi: X\dasharrow S$ be the Iitaka fibration.
Assume  $n\ge3$ and the termination of flips holds.
Let $A=\sum L_i$ be the sum in Section \ref{Setup} of nef $\bR$-divisors $L_i$ which are also $G$-eigenvectors such that $A$ is a relatively nef and big divisor.
After replacing $G$ by a finite-index subgroup and $A$ by a large multiple, then the following hold.
\begin{enumerate}
	\item There is a sequence of birational maps:
	$$
	X=X_0\overset{\tau_0}{\dasharrow} X_1 \overset{\tau_1}{\dasharrow}  \cdots \overset{\tau_{s-1}}\dasharrow X_s\overset{\tau_s}{\dasharrow} X_{s+1}=Y
	$$
	such that each $X_j\dasharrow X_{j+1}$ $(0\le j<s)$ is either a divisorial contraction of a $K_{X_j}$-negative extremal ray or a $K_{X_j}$-flip over $B$.
		\item The pair $(X_i, A_i)$ and $X_i, (0\le i\le s+1)$ have at worst klt singularities and $X_i$ is $\bQ$-factorial, where $A_i =(\tau_{i-1}\circ \cdots\circ \tau_0)_*A$.
	\item  The induced action of $G$ on each $X_i$ $(0\le i\le s+1)$ is biregular, and each $(X_i,G)$ satisfies $\mathrm{Hyp}(n,n-\kappa(X)-1)$.
\item $L_{j,i}=(\tau_{i-1}\circ \cdots\circ \tau_0)_*L_j$  is a relatively nef $\bR$-Cartier $G$-eigenvector.
Let $A_i=(\tau_{i-1}\circ \cdots\circ \tau_0)_*A$.
Then $A_i$ is a relatively nef and big $\bR$-Cartier divisor.
Further, $L_{j,i}=\tau_i^*L_{j,i+1}$ and $A_i=\tau_i^*A_{i+1}$.
\item $K_Y$ is an $\bR$-Cartier divisor and $K_{X_s}=\tau_s^*(K_Y)$.
\item $K_Y+A_Y$ is a relatively ample divisor, where  $A_Y:=A_{s+1}=\tau_{s*}A_s$.
Thus, $K_{X_s}+A_s=\tau_s^*(K_Y+A_Y)$ is a relatively nef and big divisor.
\end{enumerate}
As a result, the very general fiber $F$ of the Iitaka fibration of $\varphi_{mK_Y}$  is a klt Calabi-Yau surface or a $Q$-abelian variety.
\end{theorem}
\begin{proof}
{\bf Step 1. Existence of a scaled G-EMMP relative to the Iitaka fibration.}
Write $A=A_k+E/k$  as in Lemma \ref{lemma-A-big}, so that $(X, A)$ is klt.
Fix a relatively ample Cartier divisor $M$ such that the pair $(X, A+M)$ is nef and klt.
This is true because klt is an open condition, so we can  replace $M$ by $(\frac{1}{c})M'$ with $M'$ a general member of $|c(M+n\pi^*H_S)|$ for some large integers $c$ and $n$ and an ample divisor $H_S$ on $S$.

If $K_X+A$ is relatively nef, then $(X, A)$ is the final model and we stop the MMP.
Suppose $K_X+A$ is not relatively nef.
By Lemma \ref{lemma-A-big}, $A$ is larger than some relatively ample divisor $A_k$ for $k\gg 1$.
Then by the Cone theorem (cf. \cite[Corollary 3.8.2]{BCHM10}), 
there are only finitely many $K_X+A$-negative extremal rays $R=\bR_{\ge0}[\ell]$  generated by a rational curve in $\NE(X/S)$.
We may assume that all these $\ell$ satisfy $A\cdot \ell=0$ and $K_X\cdot\ell<0$; otherwise, after replacing $A$ by a large multiple, we would have $(K_X+A)\cdot\ell>0$ for all these finitely many $\bR_{\ge0}[\ell]$, and hence $K_X+A$ is relatively nef.
Since $A=\sum_{i=1}^{n-\kappa} L_i$ and each $L_i$ is nef, $A\cdot \ell=0$ implies that $L_i\cdot\ell=0$ for all $i$.
Notice that  $g^*K_X=K_X$ for all $g\in G$ since $\kappa>0$.
Since $L_i\cdot g_*\ell=\chi_i(g)L_i\cdot\ell=0$ and so $A\cdot g_*\ell=0$, and 
$$
(K_X+A)\cdot g_*\ell=g^*(K_X+A)\cdot\ell=(K_X+A)\cdot\ell<0.
$$ 
Thus, $g_*\ell$ is an extremal curve and  also  $A\cdot g_*\ell=0$.
Then after replacing $G$ by a finite-index subgroup, these finitely many extremal rays $\bR_{>0}[\ell]$ are fixed by $G$.

  Since our pair $(X,A)$ is $\bQ$-factorial klt and the boundary divisor $A$ is a relatively big divisor, we can apply \cite[Corollary 1.4.2]{BCHM10} to  run the $M$-scaled MMP for the pair $(X, A)$ as in \cite[\S~3.10]{BCHM10}.
  Notice that $K_X+A$ is not relatively nef, but $K_X+A+M$ is relatively nef by the choice of $M$.
There is an extremal ray $R=\bR_{\ge0}[\ell]$ such that 
$$
\lambda_0:=\inf\{\alpha\ge0~| ~K_X+A+\alpha M \text{ is relatively nef over }  S \},
$$
we have $K_X+A+\lambda_0 M$ is nef, $(K_X+A)\cdot\ell<0$ and $(K_X+A+\lambda_0 M)\cdot\ell=0$.
  As proved above, such $\ell$ satisfies $A\cdot \ell=0$ and hence $L_i\cdot\ell=0$ and $K_X\cdot\ell<0$, and we may assume that the extremal ray $R$ is $G$-stable.
Then there is a corresponding extremal contraction
$$
f:=\Contr_R: X \to Y
$$
onto a normal projective variety $Y$ such that the fibers of $f$ are connected, and a curve $C\subseteq X$ is contracted by $f$ to a point on $Y$ if and only if the class $[C]\in R$.
If $B$ is a Cartier divisor on $X$ such that $B\cdot\ell=0$, then $B=f^*B_Y$ for some Cartier divisor $B_Y$ on $Y$.
Since $R$ is $G$-stable, and $f$ is completely determined by $R$, the action of $G$ on $X$ descends to a biregular action of $G$ on $Y$ such that the morphism $f: X\to Y$ is $G$-equivariant.
Since $\kappa>0$ and $\kappa(F)=0$, we only consider the following two cases.

Case (I) (divisorial).
By Proposition \ref{Bir}, $(Y, G)$ satisfies $\mathrm{Hyp}(n,n-\kappa-1)$.
Since $L_j\cdot \ell=0$, then by the property of contraction $f$, we have $L_j=f^*L_{j,1}$ for some $\bR$-Cartier divisor $L_{j,1}$ on $Y$.
We have $L_{j,1}=f_*L_j$ by the projection formula.
Then $L_{j,1}$ is nef since $L_j$ is nef.
So $g^*L_j=\chi(g)L_j$ implies that $g^*L_{j,1}=\chi(g)L_{j,1}$.
Set $(X_1, A_1):=(Y, f_*(A))$.
We can continue the G-EMMP.

Case (II) (flip). 
$\dim Y=\dim X$ and the exceptional set $\Exc(f)$ of $f$ is a Zariski closed subset of $X$ of codimension $\ge2$ in $X$.
This $f: X\to Y$ is called a flipping contraction.
In this case, the natural map 
$$
X^+:=\Proj\bigoplus_{m\ge0}\mathcal O_Y(f_*\lfloor m(K_X+A)\rfloor) \to Y
$$
is a birational morphism such that $K_{X^+}+A^+$ is relatively ample over $Y$.
Here  $A^+$ is the proper transform of $A$.
Both birational maps $X\to Y$ and $X^+\to Y$ are isomorphisms in codimension 1.
In particular, we have $\NS_\bR(X/S)=\NS_\bR(X^+/S)$.
The map $X\dasharrow X^{+}$ is called a $(K_X+A)$-flip.
Since the flip $X^+$ is uniquely determined by the extremal ray $R$ stabilized by $G$, there is a biregular action of $G$ on $X^+$.
By Proposition \ref{Bir}, $(X^+, G)$ satisfies $\mathrm{Hyp}(n,n-\kappa-1)$.
As in Case (I), we also have $L_j=f^*f_*L_j$ and $f_*L_j$ is an $\bR$-Cartier divisor on $Y$.
Let $L_{j,1}$ be the pullback of $f_*L_j$ by the birational morphism $X^+\to Y$.
This is an $\bR$-Cartier nef divisor.
And we also have $g^*L_{j,1}=\chi(g)L_{j,1}$.
Set $(X_1,A_1):=(X^+, A^+)$.
We can continue the G-EMMP.

As proved above, we only consider the Cases (I) and (II).
Since a divisorial contraction decreases the Picard number, setting $(X, A)=(X_0, A_0)$, we may assume that we have a sequence
$$
	(X_0, A_0)\overset{\tau_0}{\dasharrow} (X_1, A_1) \overset{\tau_1}{\dasharrow}  \cdots \overset{\tau_{s-1}}\dasharrow (X_s, A_s)\overset{\tau_s}{\dasharrow} \cdots
$$
of flips $\tau_i$ corresponding to the extremal rays $\bR_{\ge0}[\ell_i]$ with $(K_{X_i} +A_i)\cdot\ell_i<0$ and $(K_{X_i}+A_i+\lambda_iM_i)\cdot \ell_i=0$.
Here $A_i$ and $M_i$ on $X_i$ are the direct images of  $A$ and $M$ respectively.
The real numbers $\lambda_i$ satisfy:
$$
\lambda_i=\inf\{ \alpha\ge0~|~K_{X_i}+A_i+\alpha M_i \text{ is relatively nef over }  S \}.
$$
Since the termination of flips holds, $K_{X_s}+A_s$ must be relatively nef for some $s>0$.

{\bf Step 2.  $K_Y+A_Y$ is a relatively ample divisor and $\tau_s$ is $G$-equivariant.}
Replacing $X$ by $X_s$, we may assume that $K_X+A$ is relatively nef and big.
By the base point free theorem (cf. \cite[Theorem 3.9.1]{BCHM10}),  there is a birational morphism $\gamma: X\to Y$ onto a normal projective variety $Y$ and a relatively ample $\bR$-Cartier divisor $B$ on $Y$ such that $K_X+A=\gamma^* B$.
By the projection formula, $B=K_Y+A_Y$ where $A_Y=\gamma_*A$.
So
$$
K_X+A=\gamma^*(K_Y+A_Y).
$$
Since $A=A_k+E/k$ with $A_k$ relatively ample, the extremal rays $\bR_{\ge0}[\ell]$ of $\gamma$ are perpendicular to the relatively nef and big divisor $K_X+A=\gamma^*(K_Y+A_Y)$.
Then for a sufficiently small positive number $\epsilon$, we have
\begin{equation*}
\begin{split}
0&=\gamma^*(K_Y+A_Y)\cdot\ell=(K_X+E/k+\epsilon A_k)\cdot \ell +(1-\epsilon)A_k\cdot\ell
\\&>(K_X+E/k+\epsilon A_k)\cdot \ell.		
\end{split}
\end{equation*}
So these are all $(K_X+E/k+\epsilon A_k)$-negative.
And there are only finitely many such $\bR_{>0}[\ell]$ since $E/k+\epsilon A_k$ is relatively big.
If $\gamma: X\to Y$ is an isomorphism, then the proof is clear.
Otherwise, there are such $\ell$.
Replacing $A$ by a large multiple, we may assume that the above finitely many extremal rays $\bR_{>0}[\ell]$ satisfy 
$$
0=(K_X+A)\cdot\ell=K_X\cdot\ell=A\cdot\ell.
$$
Since $A=\sum_{j=1}^{n-\kappa} L_j$ is a sum of nef divisors, we have $L_j\cdot\ell=0$.
Then the $G$-images $g_*\ell$ of $\ell$ are extremal rays and $K_X\cdot g_*\ell=L_j\cdot g_*\ell=0$.
Since there are only finitely many such extremal rays $\bR_{\ge0}[C]$, we may assume that all of them are $G$-stable, after replacing $G$ by a finite-index subgroup.
Then $\gamma$ is $G$-equivariant, and $(Y,G)$ satisfies $\mathrm{Hyp}(n,n-\kappa-1)$.

Now the contraction theorem (cf. \cite[Theorem 2.4.1]{Kawamata24}) and $L_i\cdot \ell=0$ for all the above extremal $\ell$ imply that $L_j=\gamma^*L(Y)_j$ for some $\bR$-Cartier divisor $L(Y)_j$ on $Y$ which is a nef common eigenvector of $G$, since so is $L_j$.
Further, $(L_Y)_j=\gamma_*L_j$ by the projection formula.
By the same contraction theorem, $K_X\cdot\ell=0$ implies that $K_Y=\gamma_*K_X$ is an $\bR$-Cartier divisor and $K_X=\gamma^*K_Y$.

{\bf Step 3. The very general fiber $F$  of the Iitaka fibration of $\varphi_{mK_Y}$ is a klt Calabi-Yau surface  or a $Q$-abelian variety.}
Theorem \ref{thm-n-kappa=2} implies \(\kappa(F)=0\). 
After shrinking $S$ to a nonempty smooth open subset $U$ over which \(\pi\) is flat, set \(Y_U:=\pi^{-1}(U)\). 
For a general fiber \(F=\pi^{-1}(s)\) with \(s\in U\), we have \(K_{Y_U}|_F\sim_{\mathbb Q}K_F\). Since \(K_Y\) is \(\mathbb Q\)-Cartier and \(Y_U\subset Y\) is open, this implies \(K_Y|_F\sim_{\mathbb Q}K_F\).
By the divisorial Zariski decomposition (cf. \cite[III, \S 1.12]{Nakayama04}), we have 
$$K_F=P+N,$$
where $P$ lies in the closed movable cone and $N$ is an effective divisor.
$$
K_F=g^*(K_F)=g^*P+g^*N
$$
and $K_F-g^*N=g^*P$ is movable, so $g^*N\ge N$ by the minimality of the negative part $N$ (cf. \cite[III, Proposition 1.14]{Nakayama04}).
Applying the above to $g^{-1}$ we get $N\ge g^*N$.
So we have
$$
g^*P\equiv P, \qquad g^*N=N.
$$
So every component of $N$ is $G$-periodic and hence stabilized by $G$.

Now 
$$
P\cdot M\cdot M_1\cdots M_{n-\kappa-2}=P|_M\cdot M_1|_M\cdots M_{n-\kappa-2}|_M\ge0
$$
for every irreducible divisor $M$ and nef $\bR$-Cartier divisor since $P|_M$ is a pseudoeffective divisor on $M$.
Notice that $K_F$ and $N$ are $G$-periodic subvarieties of $F$, and $(F,G)$ satisfies $\mathrm{Hyp}(n-\kappa,n-\kappa-1)$.
Then $A_Y|_F$ is a nef and big divisor as in \cite[Lemma 3.6]{Zhang16}.
Then by \cite[Lemma 3.7]{Zhang16}, $(K_F)\cdot (A_Y|_F)^{n-\kappa-1}=N\cdot (A_Y|_F)^{n-\kappa-1}=0$. 
So $$P\cdot (A|_F)^{n-\kappa-1}=(K_F-N)\cdot (A_Y|_F)^{n-\kappa-1}=0.$$
Thus we may apply Lemma \ref{Hodge-lem} to have $P\equiv0$.

Notice that $(K_Y+A_Y)|_F=K_F+A_Y|_F$ is an ample divisor since $K_Y+A_Y$ is relatively ample over $S$.
Then by \cite[Proposition 3.11(7)]{Zhang16} in which  set $F=X(0)$, $(A_Y|_F)|_Z\equiv0$ for every $G$-periodic proper subvariety $Z$ of $F$.
Further,
$$
(K_F+A_Y|_F)|_Z=K_F|_Z\equiv N|_Z
$$
is also nef, especially for $Z=N_i$, a component of $N$.
Thus,
$$
N\cdot M\cdot M_1\cdots M_{n-\kappa-2}=N|_M\cdot M_1|_M\cdots M_{n-\kappa-2}|_M\ge0
$$
for every irreducible divisor $M$ and nef $\bR$-Cartier divisors $M_i$ because $N|_M$ is a pseudo-effective divisor on $M$ (even when $M=N_i$); also $N\cdot (A_Y|_F)^{n-\kappa-1}=0$ by \cite[Lemma 3.7]{Zhang16}.
Lemma \ref{Hodge-lem} implies $N\equiv0$.
So the effective divisor $N=0$.
Thus $K_F \equiv N=0$.
Then the known abundance theorem in the case of zero log Kodaira dimension (cf. \cite[V. Corollary 4.9]{Nakayama04}) implies that $K_F\sim_\bQ 0$. 
We then obtain the ampleness of $A_Y|_F\equiv K_F+A_Y|_F$.
If $\dim F=2$, then $F$ is a klt Calabi-Yau surface.
If $\dim F\ge3$, then we can apply \cite[Lemma 3.9]{Zhang16} to $(F,G)$, every $G$-periodic proper subvariety of $F$ is a point.
Thus, $F$ is a $Q$-abelian variety by \cite[Lemma 3.10]{Zhang16}.
\end{proof}

\end{document}